\documentclass[12pt,a4paper]{amsart}
\usepackage{amssymb}
\usepackage{amsfonts}
\usepackage{amsmath}
\usepackage[mathscr]{eucal}

\usepackage{mathrsfs}

\usepackage{color}

\newtheorem{thm}{Theorem}[section]
\newtheorem{prop}[thm]{Proposition}
\newtheorem{lem}[thm]{Lemma}

\numberwithin{equation}{section}

\def\Z{{\Bbb Z}}
\def\Q{{\Bbb Q}}
\def\R{{\Bbb R}}
\def\C{{\Bbb C}}

\def\emp{\varnothing}

\def\fA{{\frak A}}

\def\cB{{\mathscr B}}

\def\cU{{\mathscr U}}

\def\GL{{\operatorname {GL}}}

\def\SL{{\operatorname{SL}}}

\def\PGL{{\operatorname{{\rm{PGL}}}}}

\def\Re{{\operatorname {Re}}}

\def\tr{{\operatorname{tr}}}
\def\nr{{\operatorname{N}}}

\def\Hom{{\rm{Hom}}}

\def\vol{{\operatorname{vol}}}

\def\leq{\leqslant}
\def\geq{\geqslant}
\def\bsl{\backslash}

\def\d {{{\rm d}}}

\def\1{{\bold 1}}

\newcommand{\e}{\epsilon}

\renewcommand{\t}{\tau}

\newcommand{\Scal}{{\mathscr S}}

\def\GL{{\rm {GL}}}

\renewcommand{\Re}{\operatorname{Re}}

\def\GL{\mathrm{GL}}
\def\PGL{\mathrm{PGL}}
\def\SL{\mathrm{SL}}

\def\dim{\mathop{\mathrm{dim}}}

\def\mod{\mathop{\mathrm{mod}}}

\def\lim{\mathop{\mathrm{lim}}}

\def\cInd{\mathrm{cInd}}
\def\exp{\mathop{\mathrm{exp}}}

\def\cB{\mathcal{B}}
\def\cb{\mathbf{b}}
\def\t{^\times}

\def\be{\begin{enumerate}}
\def\ee{\end{enumerate}}
\def\ha{\frac{1}{2}}
\def\iv{^{-1}}
\def\dw{$ $\\$\sp\sp\sp$}
\def\ds{\\$\sp\sp\sp$}

\def\lr{\Longleftrightarrow}
\def\sp{\text{ }}
\def\dm{\ul}

\def\al{\alpha}
\def\ep{\epsilon}

\def\vpi{\varpi}
\def\Ga{\Gamma}

\def\si{\sigma}

\def\1{\mathbf{1}}
\def\0{\mathbf{0}}
\def\ul{\underline}
\def\wh{\widehat}
\def\ol{\overline}

\def\bsl{\backslash}
\def\inf{\infty}

\begin{document}
\pagestyle{plain}
\title{Functional equations and gamma factors of local zeta functions for the metaplectic cover of ${\rm SL}_2$}
\author{Kazuki Oshita}
\address{Faculty of Mathematics and Physics, Institute of Science and Engineering, Kanazawa University, Kakumamachi, Kanazawa, Ishikawa, 920-1192, Japan }
\email{daxiayiqi@gmail.com}

\author{Masao Tsuzuki}
\address{Faculty of Science and Technology, Sophia University, Kioi-cho 7-1 Chiyoda-ku Tokyo, 102-8554, Japan}
\email{m-tsuduk@sophia.ac.jp}

\maketitle

\begin{abstract}
We introduce a local zeta-function for an irreducible admissible supercuspidal representation $\pi$ of the metaplectic double cover of $\SL_2$ over a non-archimedean local field of characteristic zero. We prove a functional equation of the local zeta-functions showing that the gamma factor is given by a Mellin type transform of the Bessel function of $\pi$. We obtain an expression of the gamma factor, which shows its entireness on $\C$. Moreover, we show that, through the local theta-correspondence, the local zeta-function on the covering group is essentially identified with the local zeta-integral for spherical functions on ${\rm PGL}_2\cong {\rm SO}_3$ associated with the prehomogenous vector space of binary symmetric matrices. 
\end{abstract}

\section{Introduction}
The automorphic $L$-functions and the zeta-funtions of prehomogenous vector spaces have their origin in the famous Tate's thesis \cite{Tate}, whose influence on the later development of the thery has been decisive. All the basic features in the modern theroy of $L$-functions already been observed in \cite{Tate}; one of such is the local zeta-integral, which should be introduced in a different way in a different context but enjoys a common property that it should be linked to its Fourier dual by a ``gamma factor" in a functional equation. 
The local zeta-integral for prehomogenous vector spaces with spherical functions are originally introduced by Fumihiro Sato (\cite{Sa94}, \cite{Sa96}), who establised their functional equations over the real number field $\R$ assuming certain nice properties of the space of the shperical functions. More recently, Wen-Wei Li conducted series of studies in a broder setting and framework, including the case when the spherical function is not necessariry unramified and the base field is non-archimedean (\cite{Li}, \cite{Li2020}). 

In this paper, we consider the local zeta-integral arising from the zeta-functions of the prehomogenous vector space $(PD^\times \times{\rm GL}_1, V)$ with coefficients related to the toric periods of automorphic forms (\cite{Sa94}), where $D$ is a quaternion algebra over an algebraic number field $F$ and $V$ is the space of elements of $D$ with reduced trace $0$. In this setting, the local zeta-integral in question is defined as a distribution on a space of spherical functions of an irreducible admissible representation $\pi$ of $k^\times\bsl (D\otimes_{F} k)^\times$ over $k$, a completion of $F$ at a place. When $k$ is non-archimedean local field, we can clarify the properties of the gamma factors of the local zeta-integral by relating them to the Mellin transformation of the Whittaker functions belonging to $\tilde \pi$, the theta-lift of $\pi$ to the double cover of $\SL_2(k)$; this is our key observation. In a context of global zeta-functions, similar point of view has been noticed by F. Sato in a broder perspective for some time. Indeed, he made a remark (\cite[p,134, Remark (2)]{Sa94}) that the global zeta-function of prehomogeneous vector spaces should be related to a Dirichlet series obtained from the Fourier coefficients of a modular form of half integral weight (cf.\cite{Sh}). We confirm that this relationship might also be true in the local context through a case study. The gamma factor of our zeta-function is defined as an improper Mellin transform of the Bessel function of $\tilde \pi$. The Bessel functions of latter object have been studied by \cite{Soudry}, \cite{Baruch}, \cite{BM} in a slghtly different but related context. In this article, to explain the main idea clearly, avoiding a subtle isse of convergence, we focus on the case when $\tilde \pi$ is supercuspidal, wheres the theory can be developed without the supercuspidality assumption on $\tilde \pi$. As a potential application, our main results can be used to obtain the convexity bound or a subconvexity boud of the global zeta functions attached to modular forms of half-integral weight. We leave the general case and the application to a futuer work.

\subsection{Description of main result}
Let us explain the main results of this paper introducing basic notation on the way. Let $k$ be a local non-Archimedean field of characteristic $0$, $O$ the ring of integers in $k$, and $P$ the maximal ideal of $O$. Throughout this article, we suppose 
\begin{align}
\text{$q:=\#(O/P)$ is odd.} 
\label{Odd}
\end{align}We fix a uniformizer $\vpi\in P$. Let $|\cdot|$ denote the normalized absolute value of $k$, i.e., $|\vpi|=q\iv$. Fix a non-trivial additive character $\psi$ of $k$ that is trivial on $O$ and non-trivial on $P^{-1}$. Let $\d t$ be the self-dual Haar measure on $k$ with respect to the self-duality defined by $\psi$; then $\vol(O)=1$. A Haar measure on the multiplicative group $k\t$ is fixed to be $d\t x:=dx/|x|$. Let $\ol{S}=\{[g,\e]\mid g\in S,\,\e\in \{\pm 1\}\}$ be the non-trivial topological double covering group of $S:=\SL_2(k)$ (see \S\ref{sec:Cover}). Let $(\pi,V)$ be an irreducible admissible representation of $\ol{S}$. For each element $\xi\in k\t$ a linear map $l^\xi:V\to\C$ is called a $\psi^\xi$-Whittaker functional of $\pi$ if it satisfies 
$$l^\xi(\pi(n(a))v)=\psi(\xi a)l^\xi(v)\quad\left(v\in V,\sp a\in k\right).$$
It is a basic theorem that the $\psi^{\xi}$-Whittaker functionals of $\pi$ are unique up to scalar multiples (\cite{GelbartHowPS1979}; see also \cite[Lemma 2]{Waldspurger1991}). Set
\begin{align}
k(\pi):=\{\xi\in k\t\mid\text{there exists a nontrivial }\psi^\xi\text{-Whittaker functional } l^\xi\text{ of }\pi.\}.
\label{kPi}
\end{align}
It is easy to see that $k(\pi)$ is a union of square classes $\xi(k^\times)^2$ so that $\#(k(\pi)/(k^\times)^2)\leq \#(k^\times/(k^\times)^2)=4$. 

In the rest of the article, we suppose that
\begin{itemize}
\item[(i)] $\pi$ is genuine, i.e., the homomorphism $\pi:\ol{S}\rightarrow {\rm GL}(V)$ does not factor through the convering morphism $\ol{S}\rightarrow S$, 
\item[(ii)] $\pi$ is supercuspidal, i.e., it satisfies the following condition:
$$\text{For all $v\in V$ there exists $n\in \Z$ such that $\int_{P^n}\pi(n(x))vdx=0$},  
$$
where we use the notation $n(x):=[\left(\begin{smallmatrix}1&x\\ 0&1\\ \end{smallmatrix}\right),1]\in \ol{S}$ for $x\in k$. 
\end{itemize} For each $v\in V$, $\xi\in k(\pi)$ and $\mu$ is a character of $k\t$ we define a local zeta-function as
$$Z(s,\mu, l^\xi,v):=2\int_{k\t}l^\xi(\pi(\dm{x})v)\chi_\psi(x)\mu(x)|x|^{s-\ha}d\t x, \quad s\in \C, $$
where $\dm{x}:=[\left(\begin{smallmatrix}x&0\\0&x\iv\end{smallmatrix}\right),1]$ and $\chi_\psi: k\t\rightarrow \C^\times$ is a function on $k^\times$ defined by the formula \eqref{Def-GenChar}. We have that the local zeta-function $Z(s,\mu,l^\xi,v)$ converges absolutely for all $s\in\C$, and $Z(s,\mu,l^\xi,v)$ is expressed as a polynomial in $q^s, q^{-s}$ (see Theorem \ref{thm-ZetaInt}). 

Let $J_{\pi}^{\xi,\eta}(g)$ be the Bessel function of $\pi$ introduced in \cite{BM}, whose defintion is recalled in \S\ref{sec:Bessel}. For any character $\mu$ of $k^\times$, we define a gamma factor $\Ga^{\xi,\eta}_{\pi,\mu}(s)$ to be the limit (see \S\ref{sec:Bessel}): 
\begin{align}
\Ga^{\xi,\eta}_{\pi,\mu}(s):=\lim_{\substack{n\rightarrow \infty \\ m\rightarrow \infty}} 2\int_{P^{-n}-P^{m}}J^{\xi,\eta}_\pi(\dm{x}w)\chi_\psi(x)\mu(x)|x|^{s-\ha}d\t x\quad\left(s\in \C\right),
 \label{Intro-f1}
\end{align}
where $w:=[\left(\begin{smallmatrix}0&-1\\1&0\\\end{smallmatrix}\right),1]$. By using a concrete realization of irreducible supercuspidal representations of $\ol{S}$ due to \cite{Man84-I}, we explicitly construct Whittaker functionals $\{l^\xi\}_{\xi\in k(\pi)}$ together with a description of a set $X(\pi)$ of representatives for $k(\pi)/(k^\times)^2$, and calculated the gamma factors. As a consequence, we obtain the following theorem, which is our first main theorem; for more precise statement, we refer to Theorem \ref{thm-gamma}, and for the proof, to \S\ref{sec:Pf-th-gamma}. 
\begin{thm}
For any $\xi,\eta\in k(\pi)$, the limit in \eqref{Intro-f1} exists and $\Ga^{\xi,\eta}_{\pi,\mu}(s)$ is expressed as a polynomial in $q^{s}$. In particular, it is entire on $\C$.
\end{thm}
Our second main theorem is the following, which is proved in \S\ref{sec:FE}. 
\begin{thm} \label{thm-FE}
 Suppose $(\pi,V)$ is an irreducible genuine supercuspidal representation of $\ol{S}$. Let $v\in V$ and $\mu$ be a character of $k\t$. Then the zeta functions satisfy the functional equation
$$Z(s,\mu, l^\xi,\pi(w)v)=\frac{1}{4}\sum_{\eta\in k(\pi)/(k^\times)^2}|\eta|\Ga^{\xi,\eta}_{\pi,\mu}(s)Z(1-s,\mu^{-1}, l^\eta,v), \quad \xi \in k(\pi).
$$
\end{thm}
In \S\ref{sec:FE}, we introduce a prehomogenous vector space $(PD^\times \times k^\times,V_D)$ associated with a quaternion algebra $D$ over $k$. Our key observation can be stated as, when $\pi\boxtimes \tilde \pi$ is an irreducible admissible representation of $P D^\times \times \ol{S}$ such that $\pi$ and $\tilde \pi$ are related by the local theta-correspondence, the local zeta-distribution with $\pi$-spherical functions (see \eqref{Def-LZDSpherical}) is directly related to the ``local zeta-functional'' $v\mapsto Z(s,\mu,l^\xi,v)$ on the $\psi^\xi$-Whittaker model of $\tilde\pi$ (see Proposition \ref{ThetaLift-L1} in detail). By using this relationship and our main results (Theorems \ref{thm-gamma} and \ref{thm-FE}), we deduce a property of the gamma matrix describing the local functional equation for the zeta-integrals with $\pi$-spherical function when $\pi$ is supercuspidal. We should remark that the functional equation itself has been known to hold by a general theorey developed in \cite{Li}.

\section{The metaplectic cover of $\SL_2(k)$ and its representations}
In this section, we introduce the metaplectic cover of $S:=\SL_2(k)$ and recall the definition of the Bessel function on the covering group $\ol{S}$. For details, we refer to \cite{BM}. 
\subsection{The mataplectic cover of $S$}\label{sec:Cover}\dw
Let $(a,b)$ $(a,b\in k^\times)$ be the Hilbert symbol of $k^\times$, i.e., 
$$(a,b):=\begin{cases}1 &\quad (\text{if $a=x^2-by^2$ for some $(x,y)\in k$}), \\
-1& \quad (\text{otherwise}).\end{cases}
$$
Set
\begin{align*}
\chi(g)&:=
\begin{cases}c &c\neq 0 \\ 
d & c=0\end{cases}
\end{align*}
for $g =\left(\begin{smallmatrix}a&b\\ c&d\\ \end{smallmatrix} \right)\in S$, and define 
$$
\{g,h\}:=\left(\frac{\chi(gh)}{\chi(g)},\frac{\chi(gh)}{\chi(h)}\right)\quad g,h\in S. 
$$
The metaplectic double cover $\ol{S}$ of $S$ is obtained by defining the group law on the set $\ol{S}:=\{[g,\ep]\mid g\in S,\sp \ep=\pm1\}$ as 
$$
[g,\ep_1][h,\ep_2]:=[gh,\{g,h\}\ep_1\ep_2]\quad([g,\ep_1],[h,\ep_2]\in\ol{S}).$$Endowed with this product, the set $\ol{S}$ becomes a totally disconected group and the first projection $\ol{S}\rightarrow S$ is a double cover. It is known that the center of $\ol{S}$ is 
$$
Z:=\{[\e 1_2, \e']\mid \e,\e'\in \{\pm 1\}\},
$$
which is a finite group of order $4$ isomorphic to $\Z/4\Z$ if $(-1,-1)=-1$ and to the Klein $4$-group $\Z/2\Z\times \Z/2\Z$ if $(-1,-1)=1$. The kernel of $\ol{S}\rightarrow S$ is $\{[1,\e']\mid \e'\in \{\pm 1\}\}$ is a subgroup of $Z$ of oeder $2$. 

In what follows, the element $[g,1]\in\ol{S}$ for $g\in S$ is denoted by $g$, and $[\left(\begin{smallmatrix} 1 & 0 \\ 0 & 1 \end{smallmatrix} \right),\ep]\in\ol{S}$ for $\ep=\pm1$ by $\ep$. For any subset $I\subset S$, define $\ol{I}:=\{[g,\ep]\in\ol{S}\mid g\in I,\sp\ep=\pm1\}$.

\subsection{Genuine characters of $\ol{A}$} \label{sec:charA}\dw
Set 
\begin{align}
\chi_{\psi}(a)=\frac{\alpha_{\psi}(1)}{\alpha_{\psi}(a)}=\frac{\alpha_{\psi}(a)}{{\alpha_\psi}(1)}(a,-1) \qquad (a\in k^\times).
\label{Def-GenChar}
\end{align} 
Note that $\chi_{\psi}(a^2)=1\,(a\in k^\times)$. We have 
\begin{align}
\chi_{\psi}(ab)=\chi_{\psi}(a)\,\chi_{\psi}(b)\,(a,b), \quad a,b\in k^\times.
\label{GenChar-f0}
\end{align}
Set $A:=\{\left(\begin{smallmatrix} a & 0 \\ 0 & a^{-1}\end{smallmatrix}\right)\mid a\in k^\times\}\,(\subset S)$. Note that $Z\subset \ol{A}$. A quasi-character $\chi$ of $\ol{A}$ is said to be genuine if $\chi|Z$ is non-trivial. The map 
\begin{align}
[{\underline a}, \e]\longmapsto \chi_{\psi}(a)\,\e
 \label{charA-f0}
\end{align}
is a genuin character of $\ol{A}$ denoted also by $\chi_{\psi}$, and $\chi_{\psi}(a)^{2}=(a,-1)$ for $a\in k^\times$, which in turn means that the character \eqref{charA-f0} is of order at most $4$. A quasi-character $\mu$ of $k^\times$ is identified with a quas-character $[{\underline t},\e]\mapsto \mu(t)$ of $\ol{A}$. Let $\fA$ denote the set of all the genuine quasi-characters of $\ol{A}$. Any element $\chi \in \fA$ is written in a unique way as $\chi_{\psi}\mu|\cdot|^{s}$ with $\mu\in \widehat{k^\times}$ and $s\in \C/2\pi\sqrt{-1}(\log q)^{-1}\Z$.
Let us fix a Haar measure on $\ol{A}$ by 
\begin{align}
\int_{\ol{A}} f(a)\,\d a=\int_{k^\times}f([{\underline a}])\,d^\times a+
\int_{k^\times}f([{\underline a},-1])\,d^\times a
\label{HaarA}
\end{align}
where $d^\times a=\zeta_{k}(1)|a|^{-1}\d a$.

\subsection{Bessel function} \label{sec:Bessel} \dw
Let $C^\inf(k)$ denote the space of locally constant functions on $k$. For $\phi\in C^\inf(k)$, we define the improper integral of $\phi$ on $k^\times$ as the limit 
$$\int^+_k\phi(x)dx:=\lim_{n\to\inf}\int_{P^{-n}}\phi(x)dx
$$
if the limit exists. 

Let $(\pi,V)$ be an irreducible admissible representation of $\ol{S}$; we suppose that $\pi$ is genuine, or equivalently 
\begin{align}
\omega_{\pi}([1,\e])=\e\quad \e\in \{+1,-1\}, 
 \label{Genuine}
\end{align}
where $\omega_\pi:Z \rightarrow \C^\times$ is the central character of $\pi$. Fix a nontrivial additive character $\psi$ of $k$. For $\xi\in k\t$, an additive character $\psi^\xi$ is defined as 
$$\psi^\xi(a):=\psi(\xi a)\quad(a\in k). $$
We fix a system of non-trivial $\psi^\xi$-Whittaker functionals $\{l^\xi\}_{\xi \in k(\pi)}$ on $V$ once and for all; then by the multiplicity one theorem for the Whittaker functionals, for each $\xi\in k(\pi)$, we have a unique map $c_\xi:k^\times \rightarrow \C^\times$ such that  
\begin{align}
l^{\xi}(\pi(\ul{a})v)=c_\xi(a)\,l^{a^2\xi}(v), \quad \xi \in k(\pi),\,a\in k^\times.
\label{Bessel-f0}
\end{align}
For $v\in V$ and $\xi \in k(\pi)$, we define a $\psi^\xi$-Whittaker function $W^\xi_v: \ol{S} \longrightarrow \C$ by
$$W^\xi_v(g):=l^\xi(\pi(g)v), \quad g\in \ol{S}.$$
Define
$$B_S:=\{n(a)\dm{b}\mid a\in k,\sp b\in k\t\},  \qquad N:=\{n(a)\mid a\in k\}.$$Let $\xi\in k(\pi),\sp \eta\in k\t$ and $g\in\ol{B}_SwN$. Then, it is known (\cite[Proposition 7.4]{BM}) that the improper integral $\int^+_kW^\xi_v(gn(x))\psi^{\eta}(-x)dx$ converges for all $v\in V$, and the linear map $v\to\int^+_kW^\xi_v(gn(x))\psi^{\eta}(-x)dx$ is a $\psi^\eta$-Whittaker functional on $V$. From this observation, we have
$$\int^+_kW^\xi_v(gn(x))\psi^{\eta}(-x)dx=0\quad(v\in  V)$$
for $\eta \in k -k(\pi)$, and has a constant $J^{\xi,\eta}_\pi(g)\in\C$ satisfying
\begin{align}
\int^+_kW^\xi_v(gn(x))\psi^{\eta}(-x)dx=J^{\xi,\eta}_\pi(g)W^\eta_v(e)\quad(v\in V)
\label{Def-BesselFt}
\end{align}
for $\eta \in k(\pi)$. The number $J^{\xi,\eta}_\pi(g)$, viewed as a function on $\ol{B}_SwN$, is called the Bessel function of $\pi$. Note that $J^{\xi,\eta}_\pi$ depends on the choice of $\{l^{\xi}\}_{\xi \in k(\pi)}$. 

Let $\xi,\eta\in k(\pi)$ and $t,u\in k\t$. Then, by \eqref{Bessel-f0}, making a change of variables, we have
\begin{align}
J^{t^2\xi,u^2\eta}_\pi(\dm{a}w)=c_{\eta}(u)c_\xi(t)^{-1}\,(u,-1)|u|^{-2}J^{\xi,\eta}_\pi(\dm{a}\hspace{1pt}\dm{t}\hspace{1pt}\dm{u}w)\quad(a\in k\t).
  \label{Bessel-f1}
\end{align}

\section{Bessel function and its Mellin transform}
 In this section, first, we briefly recall an explicit construction of supercuspidal representations on $\ol{S}$ from \cite{Man84-I}, which is crucial in our computation. Then, we define a gamma factor as the Mellin transform of the Bessel funcion, and examine the gamma factor of a supercuspidal representation by using the results of \cite{Man84-I}. 
\subsection{A construction of supercuspidal representations}\dw
Due to assumption \eqref{Odd}, all irreducible supercuspidal representations of $\ol{S}$ are constructed by the method in \cite{Man84-I}, \cite{Man84-II}. Recall the main result of \cite{Man84-I}. 

For $n\in \Z_{\geq 0}$, set
$$H_{1,n}:=\{g\in {\rm SL}_2(O) \mid g\equiv1\mod P^n\},\quad H_{2,n}:=\begin{pmatrix}0&1\\ \vpi&0\\ \end{pmatrix}\iv H_{1,n}\begin{pmatrix}0&1\\ \vpi&0\\ \end{pmatrix}.
$$
Set $H_1:=H_{1,0}={\rm SL}_2(O)$ and $H_2:=H_{2,0}$. 
Let $(\si,W)$ be a finite dimensional smooth representation of $\ol{H_i}\subset\ol{S}$. We define a compactly induced representation $\left({\rm c Ind}^{\ol{S}}_{\ol{H_i}}(\si),W'\right)$ by
$$W':=\{\phi:\ol{S}\to W\mid\phi\text{ is compactly supported on }\ol{H_i}\bsl \ol{S},\sp\phi(hg)=\si(h)\phi(g)\quad(g\in \ol{S},\sp h\in \ol{H_i})\},$$
$$[\cInd^{\ol{S}}_{\ol{H_i}}(\si)(g)\phi]\,(h):=\phi(hg)\quad(\phi\in W',\sp g,h\in \ol{S}).$$
Let $l$ be the conductor of $\si$, i.e., $l$ is the smallest positive integer for which $\si$ is trivial on $H_{i,l}$. We say that $\si$ is strongly cuspidal if $$
 \int_{P^{l-2i+1}}\si(n(x))vdx=0 \quad \text{for all $v\in V$}. 
$$

Below, we quote \cite[Theorem 1.4]{Man84-II}:
\begin{thm} \label{thm-Man84}
The compactly inducuction from an irreducible strongly cuspidal representation of $\ol{H_i}$ yields an irreducible supercuspidal representation of $\ol{S}$. Any irreducible supercuspidal representation of $\ol{S}$ is obtained by this construction from an irreducible strongly cuspidal representation of $\ol{H_i}$ for either $i=1$ or $i=2$. 
\end{thm}

\subsection{Gamma factor}\label{sec:Gamm}\dw
For $a\in k$, the Weil constant $\al(a)\in\C$ is defined by the relation: 
$$\int_k\hat{\Phi}(x)\psi(ax^2)dx=|a|^{-\ha}\al(a)\int_k\Phi(x)\psi(-a\iv x^2)dx\quad(\Phi\in C^\inf_c(k)),$$
where $\wh{\Phi}$ is the Fourier transform of $\Phi$ defined by 
$$\hat{\Phi}(a):=\int_k\Phi(x)\psi(-2ax)dx\quad(a\in k).$$

Let  $\pi$ be a genuine irreducible supercuspidal representation of $\ol{S}$. Let $\xi,\eta\in F(\pi)$ and $\mu$ a character of $k\t$. For $n\in \Z$, set
$$
\gamma_{\pi,\mu}^{\xi,\eta}(n):=2q^{-n/2}\int_{|x|=q^{n}}J^{\xi,\eta}_\pi(\dm{x}w)\chi_\psi(x)\mu(x)d\t x. 
$$
For $s\in \C$, consider the limit
\begin{align}
\Ga^{\xi,\eta}_{\pi,\mu}(s):=\lim_{\substack{n\rightarrow \infty \\ m\rightarrow \infty}} 2\int_{P^{-n}-P^m}J^{\xi,\eta}_\pi(\dm{x}w)\chi_\psi(x)\mu(x)|x|^{s-\ha}d\t x
 \label{DefGammaFactor}
\end{align}
where $w:=[\left(\begin{smallmatrix}0&-1\\1&0\\\end{smallmatrix}\right),1]$. Note that $\gamma_{\pi,\mu}^{\xi,\eta}(n)$ and $\Ga^{\xi,\eta}_{\pi,\mu}(s)$ depend on the choice of $l^{\xi}$ and $l^{\eta}$. The following is one of our main theorems. 
\begin{thm} \label{thm-gamma}
Let  $\pi$ be a genuine irreducible supercuspidal representation of $\ol{S}$. Then for $\xi,\eta \in k(\pi)$, the limit \eqref{DefGammaFactor} exists for any $s\in\C$; there exists $M\in \Z_{>0}$ such that  
\begin{align}
\Ga^{\xi,\eta}_{\pi,\mu}(s)=\sum_{n=0}^{M} \gamma_{\pi,\mu}^{\xi,\eta}(n)\,q^{ns}, \quad s\in \C,\,\xi,\eta\in X(\pi). 
\label{thm-gamma-F}
\end{align}
In particular, $\Gamma_{\pi,\mu}^{\xi,\eta}(s)$ is a polynomial in $q^{s}$. 
\end{thm}

\subsection{Proof of Theorem \ref{thm-gamma}} \label{sec:Pf-th-gamma}
Let $(\pi,V)$ be as in Theorem \eqref{thm-gamma}. By Theorem \ref{thm-Man84}, we may suppose that $(\pi,V)=(\cInd^{\ol{S}}_{\ol{H_i}}(\si),W')$ for an irreducible strong cuspidal representation $(\si,W)$ of $\ol{H_i}$ with either $i=1$ or $i=2$. Let $l$ be the conductor of $(\si,W)$. First we consider the case $i=1$ and set $H=H_{i}$. \ds
In order to prove Theorem 4, we concretely construct $V$ and the Whittaker functionals.
\begin{lem} \label{Pf-th-gamma-L0}
There exists a basis $\cB$ of $W$ that has the following property: 
\begin{align*}
&\text{For all $\cb\in \cB$ there exists a unique additive character $\psi_\cb$ of $O$}\\
&\text{such that $\si(n(a))\cb=\psi_\cb(a)\cb$ $(a\in O)$}. 
\end{align*}
The characters $\psi_\cb$ for $\cb\in \cB$ are different from each other.
\end{lem}
\begin{proof}
Since $\si$ is trivial on $H_{1,l}$, there exists $p\in \Z_{>0}$ such that $\si(n(1))^p$ is trivial. Then, the minimal polynomial of $\si(n(1))$ divides $X^p-1$, which implies that the eigenvalues of $\si(n(1))$ have no multiplicities. Thus, we can take a basis $\cB$ of $W$ satisfying
$$\si(n(1))\cb=c_\cb\cb$$
where the eigenvalues $c_{\cb}$ $(\cb\in \cB)$ are different from each other. Due to this, noting that $n(a)$ for $a\in O$ is commutative with $n(1)$, we conclude that the basis $\cB$ satisfies the condition of this lemma.
\end{proof}
From the property of $\cB$, for $\cb\in \cB$ the linear endomorphism $v\mapsto\int_O\psi_\cb(-x)\si(n(x))vdx$ of $W$ is a projection onto the one-dimensional subspace $\langle \cb\rangle$ generated by $\cb$. Denote the image of $v\in W$ under this projection by $|v|_\cb$.\ds
For $g\in\ol{S}$ and $v\in W$, let us define a vector $\phi^g_v\in W'$ by 
$$\phi^g_v(h):=\begin{cases}\si(h')v&h=h'g,\sp h'\in\ol{H}\\ 0&h\notin \ol{H}g \end{cases}\quad(h\in\ol{S}).$$
Let $K$ be a complete system of representatives for $k/O$ such that $0\in K$. Then, the set
$$\ol{S}':=\{n(t)\dm{\vpi^n}\mid t\in K,\sp n\in\Z\}$$
is a complete system of representatives for  $\ol{H}\bsl\ol{S}$. From this, we obtain a convenent basis of $W'$, 
\begin{align}
\{\phi^g_{\cb}\mid g\in\ol{S}',\sp \cb\in \cB\},
\label{Pf-th-gamma-f1}
\end{align}
to work with. By construction, we get the equation
\begin{align}
\pi(n(\vpi^{-2n}a))\phi^{n(t)\dm{\vpi^n}}_{\cb}=\psi_\cb(a)\phi^{n(t)\dm{\vpi^n}}_{\cb}\quad(a\in O,\sp n(t)\dm{\vpi^n}\in\ol{S}',\sp \cb\in \cB).
\label{Pf-th-gamma-f0}
\end{align}
Let 
\begin{align*}
\text{$X(\pi)$ be the set of $\xi\in k^\times$ such that $\psi^{\xi}|O=\psi_{\cb}$ for some $\cb\in \cB$.}
\end{align*}
The set $k(\pi)$ defined in \eqref{kPi} is explicitly determined in terms of $X(\pi)$. 

\begin{lem} \label{Pf-th-gamma-L} We have $k(\pi)=X(\pi)(k^\times)^2$. 
Moerover, for $\xi\in X(\pi)$, we have a unique $\psi^\xi$-Whittaker functional  $l^\xi:W'\to \C$ such that
$$l^\xi(\phi)= \biggl(\int_k\phi(n(x))\psi^\xi(-x)dx\biggr)\, \cb, \quad \phi\in W'.$$
\end{lem}
\begin{proof}  Let $\xi\in k(\pi)$ and $l^\xi$ a nontrivial $\psi^\xi$-Whittaker functional on $V=W'$. Since \eqref{Pf-th-gamma-f1} is a basis of $W'$, we can find $n(t)\dm{\vpi^n}\in \ol{S}'$ and $\cb\in \cB$ such that $\phi:=\phi^{n(t)\dm{\vpi^n}}_{\cb}$ satisfies
$$ l^\xi(\phi)\not=0.
$$
Then, by \eqref{Pf-th-gamma-f0},
$$\psi^\xi(\vpi^{-2n}a)l^\xi(\phi)=l^\xi(\pi(n(\vpi^{-2n}a))\phi)=\psi_\cb(a)l^\xi(\phi)\quad(a\in O).$$
Since $l^\xi(\phi)\not=0$, we obtain $\psi^\xi(\vpi^{-2m}a)=\psi^{\vpi^{-2n}\xi}(a)=\psi_\cb(a)\quad(a\in O)$, which means $\xi \in X(\pi)(k^\times)^2$. Thus $k(\pi)\subset X(\pi)(k^\times)^2$. To show the converse inclusion, let $\xi\in X(\pi)$ and define a linear map $\lambda^\xi:W'\rightarrow W$ as 
$$\lambda^\xi: \phi\mapsto\int_k\phi(n(x))\psi^\xi(-x)dx.$$
Then, the values of $\lambda^\xi$ on the basis in \eqref{Pf-th-gamma-f1} is computed as
$$\lambda ^\xi(\phi^{n(t)\dm{\vpi^n}}_{\cb'})=\begin{cases}\psi^\xi(-t)\cb&n=0\text{ and } \cb=\cb'\\0&n\neq0\text{ or }\cb\neq \cb'\\ \end{cases}\quad(n(t)\dm{\vpi^n}\in\ol{S}',\sp \cb'\in \cB).$$
This shows that $\lambda^\xi(V)=\langle \cb\rangle$ and that $\lambda^\xi\not=0$, which allow us to define a non-zero linear form $l^\xi$ on $W'$ by the relation $\lambda^\xi(v)=l^\xi(v)\,\cb$ $(v\in W')$. Since we see
$$l^\xi(\pi(n(a))\phi)=\int_k\phi(n(x+a))\psi^\xi(-x)dx=\psi^\xi(a)\int_k\phi(n(x))\psi^\xi(-x)dx\quad(a\in k,\sp \phi\in V),$$
$l^\xi$ is indeed a $\psi^\xi$-Whittaker functional on $\pi$.
\end{proof}

Now we start the proof of Theorem \ref{thm-gamma}.

By \eqref{Bessel-f1} and Lemma \ref{Pf-th-gamma-L}, we only need to consider $J^{\xi,\eta}_\pi$ for $\xi,\eta \in X(\pi)$. Now, let $\xi,\eta\in X(\pi)$ and choose $\cb,\cb'\in \cB$ such that $\psi^\xi|_O=\psi_\cb,\sp\psi^\eta|_O=\psi_{\cb'}$. Let $\mu$ be a character of $k\t$ with the conductor exponent $m \in \Z_{\geq 0}$. Set $m':=\max(l,m)$, where $l\,(\geq 1)$ is the conductor of $\si$ as above. Fix $s\in \C$, and set$$
\Ga_n(s):=\int_{\vpi^nO\t}J^{\xi,\eta}_\pi(\dm{x}w)\chi_\psi(x)\mu(x)|x|^{s-\ha}d\t x.
$$
It suffices to show that there exists $M\in \Z_{>0}$ such that
\begin{align}
\Ga_n(s)=0\quad(n<-M\text{ or }0<n).
 \label{Pf-th-gamma-f2}
\end{align}
Indeed, if this is the case, since $X(\pi)$ is finite we can choose $M>0$ large enoug so that the equality in \eqref{thm-gamma-F} holds for all $\xi,\eta\in X(\pi)$. 

Set $\phi:=\phi^e_{\cb'}$. From the proof of Lemma \ref{Pf-th-gamma-L}, $W^\eta_\phi(e)=l^\eta(\phi)=1$; by this, we have 
\begin{align*}
J^{\xi,\eta}_\pi(g)&=J^{\xi,\eta}_\pi(g)\, W^\eta_\phi(e)=\int^+_kW^\xi_\phi(gn(y))\psi^\eta(-y)dy\\
                        &=\int^+_kl^\xi(\pi(gn(y))\phi)\psi^\eta(-y)dy=\int^+_k\int_k\phi(n(z)gn(y))\psi(-\xi z-\eta y)dzdy.
\end{align*}
Hence, for $n\in\Z$
\begin{align*}
\Ga_n(s)&=q^{n(\ha-s)}\int_{\vpi^nO\t}\int^+_k\int_k\phi(n(z)\dm{x}wn(y))\psi(-\xi z-\eta y)\chi_\psi(x)\mu(x)dzdyd\t x.
\end{align*}
The support condition of $\phi$ yields 
$$n(z)\dm{x}wn(y)\in\ol{H}\lr\begin{cases}x\in\vpi^nO\t,\sp n<0,\sp y\in\vpi^nO\t,\sp z\equiv x^2y\iv\mod O\\x\in O\t,\sp y,z\in O\\ \end{cases}.$$
This already shows that 
\begin{align}
\text{$J_{\pi,\mu}^{\xi,\eta}(\ul{x}w)=0$ for $x\in P$}
\label{J-hyouka1}
\end{align}
 as well as $\Gamma_{n}(s)=0$ for $n>0$. In the reamining part, we assume $n\leq0$. For $n\leq -l$ and $x\not\in P^{n}$, we obtain 
\begin{align*}
J_{\pi}^{\xi,\eta}(\ul{x} w)&=\int_{\varpi^n O^\times}\int_{x^2y^{-1}+O}
 \si(n(z)\ul{x} w n(y))\cb' \psi^\xi(-z-\xi^{-1}\eta y)\,dzdy
\\
&=\int_{\varpi^nO} \int_{O} \si\left(n(z)\ul{xy^{-1}}\,\left(\begin{smallmatrix} 1 & 0 \\ y^{-1} & 1 \end{smallmatrix} \right) \right)\,\cb'\,(x^{-1}y,y^{-1})\,\psi^{\xi}(-z-x^2y^{-1}-\xi^{-1}\eta y)\,dzdy
\end{align*}
Due to $n\leq -l$, 
$$
\int_{O} \si\left(n(z)\ul{xy^{-1}}\,\left(\begin{smallmatrix} 1 & 0 \\ y^{-1} & 1 \end{smallmatrix} \right) \right)\cb'\psi^{\xi}(-z)=\int_{O}\si(n(z))\,(\si(\ul{xy^{-1}})\cb')\,\psi^\xi(-z)=|\si(\ul{xy^{-1}})\cb')|_{\cb}.
$$
Thus, we obtain
\begin{align}
J_{\pi}^{\xi,\eta}(\ul{x}w)=\int_{\varpi^n O}|\si(\ul{xy^{-1}}\cb'|_\cb\,(x^{-1}y,y^{-1})\,\psi^\xi(-x^2y^{-1}-\xi^{-1}\eta y)\,dy, \quad (n\leq -l,x\not\in P^{n}). 
\label{J-sekibunkoushiki}
\end{align} 
Substituting this to the last formula of $\Ga_n(s)$, we have   
\begin{align*}
\Ga_n(s)&=q^{n(\ha-s)}\int_{\vpi^nO\t}\int_{\vpi^nO\t}|\si(\dm{xy\iv})\cb'|_\cb(x\iv y,y\iv)\psi^\xi(-x^2y\iv-\xi\iv\eta y)\chi_\psi(x)\mu(x)dyd\t x\\
&=q^{n(\frac{3}{2}-s)}\int_{O\t}\int_{O\t}|\si(\dm{xy\iv})\cb'|_\cb(x\iv y,\vpi^{-n}y\iv)\\
&\hspace{105pt}\psi^\xi(-\vpi^n(x^2y\iv+\xi\iv\eta y))\chi_\psi(\vpi^nx)\mu(\vpi^nx)dyd\t x\\
&=q^{n(\frac{3}{2}-s)}\int_{O\t}\int_{O\t}|\si(\dm{x})\cb'|_\cb(x\iv,\vpi^{-n}y\iv)\psi^\xi(-\vpi^ny(x^2+\xi\iv\eta))\chi_\psi(\vpi^nxy)\mu(\vpi^nxy)dyd\t x.
\end{align*}
From the definition of $m'$, $a\to\si(\dm{a})$ and $\mu$ are constant on $1+P^{m'}$. Due to \eqref{Odd}, the Hilbert symbol $(\cdot,\cdot)$ and $\chi_\psi$ are constant on $1+P$. Hence, 
\begin{align*}
\Ga_n(s)=q^{n(\frac{3}{2}-s)}\sum_{t,u\in O\t/(1+P^{m'})}c_{t,u}\int_{t(1+P^{m'})}\int_{u(1+P^{m'})}\psi^\xi(-\vpi^ny(x^2+\xi\iv\eta))dyd\t x\\
\text{where we set $c_{t,u}:=|\si(\dm{t})\cb'|_\cb(t\iv,\vpi^{-n}u\iv)\chi_\psi(\vpi^ntu)\mu(\vpi^ntu)$ for $t,u\in O\t/(1+P^{m'})$}.
\end{align*}
Now we examine the integral
\begin{align}
\int_{t(1+P^{m'})}\int_{u(1+P^{m'})}\psi^\xi(-\vpi^ny(x^2+\xi\iv\eta))dyd\t x, \qquad t,u\in O\t/(1+P^{m'}.
 \label{doubleInt}
\end{align}
Suppose $n\leq l-2m'$. Set
$$U_t:=\{a\in t(1+P^{m'})\mid a^2+\xi\iv\eta\in P^{l-n-m'}\}.$$
Then,
\begin{align*}
&\hspace{15pt}\int_{t(1+P^{m'})}\int_{u(1+P^{m'})}\psi^\xi(-\vpi^ny(x^2+\xi\iv\eta))dyd\t x\\
&=\int_{t(1+P^{m'})}\int_{P^{m'}}\psi^\xi(-\vpi^nu(1+y)(x^2+\xi\iv\eta))dydx\\
&=\int_{t(1+P^{m'})}\psi^\xi(-\vpi^nu(x^2+\xi\iv\eta))\left\{\int_{P^{m'}}\psi^\xi(-\vpi^nuy(x^2+\xi\iv\eta))dy\right\}dx.
\end{align*}
From the definition of $U_t$, for $x\in t(1+P^{m'})$, we have
$$\int_{P^{m'}}\psi^\xi(-\vpi^nuy(x^2+\xi\iv\eta))dy=\begin{cases}q^{m'}&x\in U_t\\ 0&x\notin U_t\\ \end{cases}.$$
Hence the integral \eqref{doubleInt} equals
\begin{align*}
&q^{-m'}\int_{U}\psi^\xi(-\vpi^nu(x^2+\xi\iv\eta))dx\\
&=q^{-m'}\sum_{c\in U/(1+P^{l-n-m'})}\int_{c(1+P^{l-n-m'})}\psi^\xi(-\vpi^nu(x^2+\xi\iv\eta))dx\\
&=q^{-m'}\sum_{c\in U/(1+P^{l-n-m'})}\psi^\xi(-\vpi^nu(c^2+\xi\iv\eta))\int_{P^{l-n-m'}}\psi^\xi(-\vpi^nc^2u(2x+x^2))dx.
\end{align*}
Since $2l-n-2m'\geq l$ and $l-m'<l$, we obtain
$$\int_{P^{l-n-m'}}\psi^\xi(-\vpi^nc^2u(2x+x^2))dx=\int_{P^{l-n-m'}}\psi^\xi(-2\vpi^nc^2ux)dx=0.$$
By the consideration so far, we have $\Ga_n(s)$ when $l-2m'\geq n$ completing the proof of \eqref{Pf-th-gamma-f2}. 

The case when $H=H_2$ is similar.

\subsection{An example}
We give an example in which the gamma factor dose not vanish. Set $k:=\Q_3$, $O:=\Z_{3}$, $P=3\Z_3$. Put $[a]:=\sum^{-1}_{i=n}a_i3^i$ for $a=\sum^{\inf}_{i=n}a_i3^i\in k$ and define an additive character $\psi$ of $k$ by  
$$\psi(a):=e^{2\pi i[a]}\quad(a\in k).$$
Let $\si$ be a one-dimensional representation of $H_1={\rm SL}_2(\Z_3)$ defined by$$\si(w)=1,\sp\si(n(a))=\exp(2\pi i[3\iv a])\quad(a\in O),\sp.$$
Then, $\si$ is a strongly cuspidal representation of conductor $1$. 
Let $\pi:=\cInd^{\ol{G}}_{\ol{H}_1}(\si)$. Since $X(\pi)=3\iv(1+P)$, it suffices to consider the case $\xi=3\iv$. For $\mu$ be the character of $k\t$ of the conductor $O$, by using \eqref{J-sekibunkoushiki}, the gamma factor of $\pi$ can be calculated as follows.
\begin{align*}
\Ga^{\xi,\xi}_{\pi,\mu}(s)&=2\int_{\Z_3\t}J^{\xi,\xi}_\pi(\dm{x}w)\chi_\psi(x)\mu(x)d\t x\\
&=2\int_{\Z_3\t}\int_{\Z_3}\int_{\Z_3}\si(n(z)\dm{x}wn(y))\psi(-\xi z-\xi y)\chi_\psi(x)\mu(x)dzdyd\t x=\tfrac{4}{3}.
\end{align*}
Note that this is a non-zero constant.

\section{Local zeta function and functional equation} \label{sec:FE}
Let $(\pi,V)$ be an irreducible genuine supercuspidal representation of $\ol{S}$. Let $\mu$ be a character of $k^\times$. For $\xi\in k(\pi)$ and $\sp v\in V$,  we define  
\begin{align}
Z(s,\mu,l^\xi,v):=2\int_{k\t}W^\xi_v(\dm{x})\chi_\psi(x)\mu(x)|x|^{s-\ha}d\t x 
\label{Def-LZF}
\end{align}
for $s\in \C$, and call this the zeta function, where $\{l^\xi\}_{\xi\in k(\pi)}$ is a family of non-trivial $\psi^{\xi}$-Whittaker functionals on $V$. 

\begin{thm} \label{thm-ZetaInt}
The local zeta function $Z(s,\mu,l^\xi,v)$ converges for any $s\in \C$ and $Z(s,\mu,l^\xi,v)$ is expressed as a polynomial in $q^s,\sp q^{-s}$
\end{thm}
\begin{proof}
For $\xi\in k(\pi)$ and $v\in V$, we define $K^\xi_v:k\t\to\C$ by
$$K^\xi_v(a):=W^\xi_v(\dm{a})\quad(a\in k\t).$$
It suffices to show that $K^\xi_v$ is compactly supported on $k\t$. Since $v\in V$ is a smooth vector, there exists $N\in \Z_{>0}$ such that $\pi(n(t))v=v\,(t\in P^{N})$. Then, for $t\in P^N$, 
$$K^\xi_v(a)=K^\xi_{\pi(n(t))v}(a)=W^\xi_v(\dm{a}n(t))=W^\xi_v(n(a^2t)\dm{a})=\psi^\xi(a^2t)K^\xi_v(a)\quad(a\in k\t).$$
Let $t_0\in k$ be such that $\psi^\xi(t_0)\not=1$. Suppose $a^2 \not\in t_0 P^{-N+1}$; then $t=a^{-2}t_0\in P^{N}$ and hence $K^\xi_v(a)=\psi^\xi(a^2t)K^\xi_v(a)=\psi^{\xi}(t_0)K^\xi_v(a)$. Since $\psi^{\xi}(t_0)\not=1$, we have $K^\xi_v(a)=0$. Thus, $K_v^\xi(a)$ is identically zero when $|a|$ is large. \ds
Next, let us show that $K_v^\xi(a)$ is zero when $|a|$ is small. Since $\pi$ is supercuspidal, there exists $n\in \Z$ such that 
$$\int_{P^n}\pi(n(x))vdx=0.
$$
When $|a|$ is small enough so that $\psi^{\xi}(a^2x)=1$ for all $x\in P^{n}$, we have
$$0=l^\xi\left(\pi(\dm{a})\int_{P^n}\pi(n(x))vdx\right)=\int_{P^n}W^\xi_v(\dm{a}n(x))dx=K^\xi_v(a)\int_{P^n}\psi^\xi(a^2x)dx=\vol(P^{n})K^{\xi}_v(a),
$$
which implies $K^\xi_v(a)=0$. 
\end{proof}

For $a\in k\t$, define $\phi_a:k\to\C$ as
$$\phi_a(x):=W^\xi_v(\dm{a}wn(x))\quad(x\in k).$$
\begin{lem} \label{thm-FE-L1}
We have $\phi_{a}\in {\mathcal S}(k)$. 
\end{lem}
\begin{proof} Evidently, $\phi_a$ is locally constant on $k$. Hence it suffices to show that its support is bounded. Since $v\in V$ is a smooth vector, there exist finite number of smooth functions $c_i:{\rm SL}_2(O) \rightarrow \C$ $(1\leq i\leq m)$ and vectors $v_i\in V$ $(1\leq i \leq m)$ such that 
$$
W_{\pi(k)v}=\sum_{i=1}^{m}c_i(k)\,W_{v_i}, \quad k \in {\SL}_2(O).
$$
Suppose $|x|>1$. Then, 
$$
w n(x)=\left(\begin{smallmatrix}1 & -x^{-1} \\ 0 & 1 \end{smallmatrix} \right)\left(\begin{smallmatrix}x^{-1} & 0 \\ 0 & x \end{smallmatrix} \right)k_x \quad (k_x:=\left(\begin{smallmatrix}1 & 0 \\ x^{-1} & 1 \end{smallmatrix} \right) \in \SL_2(O))
$$
is the Iwasawa decomposition of $wn(x)$. Hence, 
\begin{align*}
\phi_a(x)=W^{\xi}_{v}=\psi^\xi(-a^2x)\,W_{\pi(k_x)v}^\xi(\ul{ax^{-1}})
=\psi^{\xi}(-a^2x)\sum_{i=1}^{m}c_{i}(k_x)\,W_{v_i}^{\xi}(\ul{ax^{-1}}). 
\end{align*} 
As shown in the proof of Theorem \ref{thm-ZetaInt}, the function $u\mapsto W_{v_i}^{\xi}(\ul{u})$ is of compact support on $k^\times$; hence, there exists $M>0$ such that $W_{v_i}^{\xi}(\ul{ax^{-1}})=0$ ($1\leq i \leq m)$ unless $|ax^{-1}|>M$. Therefore, the support of $\phi_a$ is contained in $|x|<|a|M^{-1}$.
\end{proof}

Recall the Bessel function $J_{\pi,\mu}^{\xi,\eta}(g)$ (see \S\ref{sec:Bessel}).\begin{prop} \label{Bessel-asymp}
We have
\begin{align*}
J_{\pi}^{\xi,\eta}(\ul{x}w)=0 \quad x\in O-\{0\}.
\end{align*}
There exists a constant $C>0$ such that  
$$
 |J_{\pi}^{\xi,\eta}(\ul{x}w)|\leq C\,\max(1,|x|), \quad x\in k^\times.
$$
If $\Re(s)<-1/2$, then 
$$
 \int_{k^\times}|J_{\pi,\mu}^{\xi,\eta}(\ul{x}w)|\,|x|^{\Re(s)-1/2}\d^\times x<\infty
$$
and 
$$
\Gamma_{\pi,\mu}^{\xi,\eta}(s)=\int_{k^\times} J_{\pi}^{\xi,\eta}(\ul{x}w)\chi_{\psi}(x)\mu(x)|x|^{s-1/2}\,d^\times x \quad (\Re(s)<-1/2). 
$$
\end{prop}
\begin{proof}
The first assertion was proved in the proof of Theorem \ref{thm-gamma} (see \eqref{J-hyouka1}). The estimate of $J_{\pi,\mu}^{\xi,\eta}(\ul{x}w)$ for $|x|\gg 1$ results from the integral expression \eqref{J-sekibunkoushiki} easily.  By the first and the second claim, the function $J_{\pi}^{\xi,\eta}(\ul{x}w)|x|^{\Re(s)-1/2}$ has a majornat $\max(1,|x|)^{\Re(s)+1/2}$ on $k^\times$, which for $\Re(s)<-1/2$ is integrable on $k^\times$; then, the last formula follows from \eqref{DefGammaFactor}. 
\end{proof}

\begin{lem} \label{thm-FE-L2}
Let $\xi \in k'(\pi)$. Then for all $a\in k^\times$, 
$$
W^\xi_v(\dm{a}w)=\sum_{\eta\in k'(\pi)}\frac{|\eta|}{2}\int_{k\t}J^{\xi,\eta}_\pi(\dm{ay}w)(ay,y)W^\eta_v(\dm{y})d\t y.
$$
\end{lem}
\begin{proof}
For $\phi \in {\mathcal S}(k)$, define its Fourier transform $\wh{\phi}$ as
$$\wh{\phi}(y):=\int_k\phi(x)\psi(-yx)dx\quad(y\in k)$$ 
where $dx$ is the self-dual Haar measure with respect to $\psi$. Fix $a\in k^\times$. By Lemma \ref{thm-FE-L1} and by the Fourier inversion formula, 
{\allowdisplaybreaks\begin{align*}
W^\xi_v(\dm{a}w)&=\wh{\wh{\phi_a}}(0)=\int_k\int_kW^\xi_v(\dm{a}wn(z))\psi(-yz)dzdy\\
&=\sum_{\eta\in k\t/(k\t)^2}\int_{\eta(k\t)^2}\int_kW^\xi_v(\dm{a}wn(z))\psi(-yz)dzdy\\
&=\sum_\eta\frac{|\eta|}{2}\int_{k\t}\int_kW^\xi_v(\dm{a}wn(z))\psi(-\eta y^2z)|y|dzdy\\
&=\sum_\eta\frac{|\eta|}{2}\int_{k\t}\int_kW^\xi_v(\dm{a}wn(y^{-2}z))\psi^\eta(-z)dzd\t y\\
&=\sum_\eta\frac{|\eta|}{2}\int_{k\t}\int_kW^\xi_{\pi(\dm{y})v}(\dm{ay}wn(z))(ay,y)\psi^\eta(-z)dzd\t y.
\end{align*}}From \eqref{Def-BesselFt}, 
$$\int_kW^\xi_{\pi(\dm{y})v}(\dm{ay}wn(z))\psi^\eta(-z)dz=\begin{cases}J^{\xi,\eta}_\pi(\dm{ay}w)W^\eta_v(\dm{y})&\eta\in k(\pi)\\ 0&\eta\notin k(\pi)\\ 
\end{cases}.
$$
Note that the improper integral in \eqref{Def-BesselFt} can be replaced with the usual integral, which is absolutely convergent due to Lemma~\ref{thm-FE-L1}. 

Applying this, we have the desired formula.
\end{proof}

Now, we start the proof of Theorem \ref{thm-FE}. 
Suppose $\Re(s)<-1/2$. Then,  
{\allowdisplaybreaks\begin{align}
Z(s,\mu, l^\xi, \pi(w)v)
&=2\int_{k\t}W^\xi_v(\dm{x}w)\chi_\psi(x)\mu(x)|x|^{s-\ha}d\t x
 \notag
\\
&=\int_{k\t}\sum_{\eta\in k(\pi)/(k\t)^2}|\eta|\int_{k\t}J^{\xi,\eta}_\pi(\dm{xy}w)(xy,y)W^\eta_v(\dm{y})\chi_\psi(x)\mu(x)|x|^{s-\ha}d\t yd\t x
 \notag
\\
&=\sum_\eta|\eta|\iint_{k\t\times k\t} J^{\xi,\eta}_\pi(\dm{x}w)(x,y)W^\eta_v(\dm{y})\chi_\psi(xy\iv)\mu(xy\iv)|xy\iv|^{s-\ha}d\t x d\t y
 \notag
\\
&=\sum_\eta|\eta|\iint_{k\t\times k\t}J^{\xi,\eta}_\pi(\dm{x}w)(x,y)W^\eta_v(\dm{y})\chi_\psi(xy\iv)\mu(xy\iv)|xy\iv|^{s-\ha}d\t x d\t y
 \label{++}
\\
&=\sum_\eta|\eta|\int_{k\t}J^{\xi,\eta}_\pi(\dm{x}w)\chi_\psi(x)\mu(x)|x|^{s-\ha}d\t x\int_{k\t}W^\eta_v(\dm{y})\chi_\psi(y)\mu\iv(y)|y|^{\ha-s}d\t y
 \notag
\\
&=\frac{1}{4}\sum_\eta|\eta|\Ga^{\xi,\eta}_{\pi,\mu}(s)Z(1-s,\mu\iv, l^\eta,v).
 \notag
\end{align}}Note that the double integral in \eqref{++} is absolutely convergent for $\Re(s)<-1/2$ by Lemma \ref{Bessel-asymp}, due to which, we apply Fubini's theorem to guarantees the above computation. Thus the functional equation in Theorem \ref{thm-FE} is proved for $\Re(s)<-1/2$. By Theorem \ref{thm-gamma} and \ref{thm-ZetaInt}, it holds for all $s\in \C$ by analytic continuation. 

\medskip
\noindent
{\bf Remark} : Let $\omega_{\pi}$ be the central character of $\pi$. Then, we have $Z(s,\mu,l^\xi,v)=0$ for all $s$, $\xi$ and $v$ unless $\omega_{\pi}(-1)=\chi_{\psi}\mu(-1)$. Indeed, $W_{v}^{\xi}(\ul{(-x)})=W_{v}^\xi([-1,(x,-1)]\,\ul{x})=\omega_{\pi}(-1)(x,-1)\,W_{v}^\xi(\ul{x})$ due to \eqref{Genuine}. Hence, by the change of variables $x\rightarrow -x$ and by \eqref{GenChar-f0}, we have
\begin{align*}
Z(s,\mu,l^\xi,v)&=\int_{k^\times}W_{v}^{\xi}(\ul{(-x)})\,\chi_\psi(-x)\mu(-x)\,|-x|\,\d x
\\
&=\omega_{\pi}(-1)\chi_{\psi}(-1)\mu(-1) \int_{k^\times} W_{v}^{\xi}(\ul{x})\,\chi_\psi(x)\mu(x)\,|x|\,\d x
=\omega_{\pi}(-1)(\chi_{\psi}\mu)(-1)\,Z(s,\mu,l^\xi,v).
\end{align*}
When $\omega_{\pi}=(\chi_\psi\mu)|Z$, by \eqref{HaarA}, we can write 
$$Z(s,\mu,l^\xi,v)=\int_{\ol{A}}W_v^{\xi}(a)(\chi_{\psi}\mu|\cdot|^{s-1/2})(a) \d a.
$$
This explains the factor $2$ in \eqref{Def-LZF}. 

\section{Local prehomogenous zeta-distribution with spherical functions}

We consider the case when the prehomogenous vector space is a $k$-form of $({\rm PGL}_2\times {\rm GL}_1, {\rm sym}^2)$ with ${\rm sym}^2$ being the space of symmetric matrices of degree $2$. This example is already examined in details over $\R$ in \cite{Sa94}, when the spherical function is for an unramified principal series of ${\rm PGL}_2(\R)$. In this section, we denote a representation of $\ol{S}$ by $\tilde \pi$, reserving the notation $\pi$ for a representation of ${\rm PGL}_2(k)$ (or its inner $k$-form).

\subsection{Local zeta-function} \label{sec:LocZeta}
Let $k_0=k(\sqrt{u})$ be an unramified quadratic extension of $k$, where $u$ is a non-square unit of $O$; the Galois conjugate of $x\in k_0$ is denoted by $\bar x$. Fix a prime element $\varpi$ of $k$. Then 
$${\rm H}=\left\{\left[\begin{smallmatrix} x & \varpi \bar y\\ y & \bar x\end{smallmatrix} \right]\mid x,\,y\in k_0 \right\}
$$
is a division subalgebra of ${\rm M}_2(k_0)$; recall that ${\rm H}$ is the unique division quaternion algebra over $k$ up to isomorphism.   

Let $D$ be a quaternion algebra over $k$ with the reduced norm map $\nr_{D}:D\rightarrow k$ and the reduced trace map $\tr_{D}:D\rightarrow k$. Then $(V_D,-\nr_{D})$ with $V_{D}:=\{X\in D \mid \tr_{D}(X)=0\}$ is a $3$-dimensional quadratic space. The quadratic space $V_{D}$ is endowed with the self-dual Haar measure $\d X$ so that $\mathcal F\mathcal Ff(0)=f(0)\,(f\in  \Scal(V_D)$ with 
$$
\mathcal F f(x)=\int_{V}f(y)\psi(\tr_{D}(xy))\,\d y. 
$$
For $g\in D^\times$, define $\rho(g)\in \GL(V_{D})$ (acting on $V_D$ from right) by 
$$
g\,\rho(g)=g^{-1}Xg, \quad X\in V_{D}.
$$
Then the map $g\mapsto \rho(g)$ induces a group isomorphism from $PD^\times=D^\times /k^\times$ onto the special orthogonal group ${\rm SO}(V_{D})$. For $\xi\in k^\times$, set
$$
V_{D}^{0}(\xi):=\{X\in V_D\mid -\nr_{D}(X)=\xi\}.
$$
Let $k_{D}^\times$ denote $k^\times$ or $k^\times-(k^\times)^2$ according as $D$ is split or not. Then $V_{D}^{0}(\xi)\not=\emp$ if and only if $\xi \in k^\times_{D}$. For $\xi \in k_{D}^\times$, Witt's theorem implies that $V_{D}^0(\xi)$ is a $PD^\times$-orbit; we choose a base point $x_{\xi}\in V_{D}^0(\xi)$ once and for all and set 
$$
T_{\xi}=\{g\in D^\times \mid x_\xi\,\rho(g)=x_{\xi}\}.
$$
Then, the mapping $T_\xi g\mapsto x_{\xi} \rho(g)$ yields a $PD^\times$-equivariant isomorphism $T_{\xi}\bsl D^\times \longrightarrow V_D^{0}(\xi)$. Note that $Z\subset T_\xi$ and that $T_{\xi}$ coincides with $k(x_{\xi})^{\times}$, where $k(x_{\xi})\subset D$ is the centralizer of $x_{\xi}$ in $D$ which is isomorphic to the quadratic etale $k$-algebra $k[\sqrt{\xi}]$. 

Let $\pi$ be an irreducible smooth representation of $PD^\times$. Let $k(\pi)$ be the set of all those $\xi\in k^\times$ such that 
$$
 \Hom_{T_\xi}(\pi,{\bf 1})\not=\{0\}.
$$ 
Note that $k(\pi)$ is a (possibly empty) union of square classes $\xi(k^\times)$ in $k^\times$. The multiplicity free theorem $\dim_{\C}\Hom_{T_\xi}(\pi,{\bf 1})\leq 1$ is known (\cite{Waldspurger1980}, \cite{Waldspurger1991}). When $k(\pi)$ is non empty, $\pi$ is said to be spherical; since $1\in k(\pi)$ when $D={\rm M}_2(k)$, any irreducible smooth generic representation $\pi$ of $\PGL_2(k)$ is spherical. For a spherical $\pi$ and $\xi\in k(\pi)$, let ${\cU}(\pi,\xi)$ denote the unique realization of $\pi$ as a submodule of smooth functions $C^\infty(V_{D}^{0}(\xi))$, i.e., ${\cU}(\pi,\xi)$ is the space of functions $g\mapsto \ell(\pi(g)v)$ with $v\in V_\pi$ and $\ell \in \Hom_{T_\xi}(\pi,{\bf 1})$. By extending the action $\rho$ of $PD^\times$ on $V_D^0$ to the action $\hat \rho$ of the direct product ${\mathcal G}:=k^\times \times PD^\times$ by letting $k^\times$ by homotheties on $V_{D}^0$, one obtaines a prehomogenous vector space $({\mathcal G} ,V_{D})$ whose ($k$-valued points of the) singular set $\Sigma$ is the zero locus of the unique basic relative invariant $P(X):=-4 \nr_{D}(X)$ on $V_{D}$. Note that the rational character of ${\mathcal G}$ attached to $P$ is $\omega(a,g): =a^{2}$. For any $\xi\in k_{D}^\times$, the set $V_{D}(\xi):=k^\times \,V_{D}^{0}(\xi)$ is an ${\mathcal G}$-orbit, which is open in $V_D$. We have 
\begin{align}
V_{D}-\Sigma=\bigsqcup_{\xi \in k_D^\times}V_{D}(\xi).
\label{OrbitDec}
\end{align}
Let $\d\tau_{PD^\times}(g)$ be the Tamagawa measures on $PD^\times$, i.e., $\d\tau_{PD^\times}$ is the quotient measure of $\d \tau_{D^\times}(g):=\d \lambda_{D}(g)|\nr_D(g)|^{-2}$ on $D^\times$ by the Haar measure on $Z\cong k^\times$, where $\lambda_{D}$ is the self-dual Haar measure on $D$. Since $|P(X)|^{-3/2}\d X$ is a $k^\times \times PD^\times$-invariant Haar measure on $V_{D}$, one can fix a Haar measure $\d\mu_{\xi}$ on the stabilizer ${\mathcal G}_{\xi}:=(k^\times \times PD^\times)_{x_\xi}$ in such a way that 
\begin{align}
\int_{V_D(\xi)}f(X)\,\frac{\d X}{|P(X)|^{3/2}}=\int_{{\mathcal G}_\xi \bsl {\mathcal G}}f(x_\xi \rho(\hat g))\frac{\d \tau_{{\mathcal G}}}{\d\tau_{\xi}}(\hat g), \quad f\in \Scal(V_{D}), 
\label{LocalZeta}
\end{align}
where $\hat g$ denotes a general element $(a,g)\in {\mathcal G}$ and $\d\tau_{{\mathcal G}}( \hat g)=\d^\times a \, \d \tau_{PD^\times}(g)$. Since $Z\bsl T_{\xi}$ is isomorphic to the identity Zariski connected component of\footnote{${\mathcal G}_\xi \cong O(\nr_{k[\sqrt{\xi}]/k})$ and $T_\xi\cong {\rm SO}(\nr_{k[\sqrt{\xi}]/k})$.} ${\mathcal G}_{\xi}=\{(\e,g)\in {\mathcal G}\mid x_\xi\,\rho(g)=\e x_\xi,\,\e\in \{\pm 1\}\}$, the measure on ${\mathcal G}_\xi$ yields one on $Z\bsl T_{\xi}$ by restriction.
Let $\frac{\d \tau_{PD^\times}}{\d \mu_\xi}$ denote the quotent measure on $(Z\bsl T_\xi)\bsl PD^\times=T_\xi \bsl D^\times $. Then the local zeta-integral attached to $u_{\xi} \in \cU(\pi,\xi)$ and $\Phi \in \Scal(V_{D})$ is defined as
\begin{align}
{\bf Z}_{\xi}(\mu|\cdot|^s; \Phi, u_{\xi}):=
|4\xi|^{s-3/2}\mu(-4\xi)\,\int_{k^\times}\int_{T_\xi \bsl D^\times}u_{\xi}(x_\xi \rho(g))\,\mu(a^2)|a|^{2s}\Phi(x_{\xi}\,\hat \rho(a,g))\,\d^\times a \,\tfrac{\d\tau_{PD^\times}}{\d\mu_\xi}(\dot{g})
 \label{Def-LZDSpherical}
\end{align}
for $s\in \C$ and $\mu \in \widehat{k^\times}$, which is shown to be absolutely convergent for $\Re(s) \gg 0$. We have an expression as an orbital integral on $V_{D}(\xi)$: 
\begin{align*}
{\bf Z}_{\xi}(\mu|\cdot|^s;\Phi,u_\xi)=\int_{V_{D}(\xi)} u_{\xi}(X)\,|P(X)|^{s-3/2}\mu(P(X))\,\Phi(X)\,\d X,\end{align*}
where the function $u_{\xi}$ on $V_{D}^{0}(\xi)$ is viewed as a homothethy invariant function on $V_{D}(\xi)$.

\subsection{Weil representation}
We refer to \cite{RallisSchiffmann} for details. Let $r_{\psi,D}$ denote the Weil representation of $\ol{S}$ on the space of Schwartz-Bruhat functions $\Scal(V_{D})$. Then,  
\begin{align}
r_{\psi,D}([{\underline t},\e])f(x)&=\e\,
\frac{\alpha_{\psi}(1)}{\alpha_{\psi}(t)}\,|t|^{3/2}\,f(tx), \quad t\in k^\times, \label{WeilR-f1} 
\\ 
r_{\psi,D}([{n}(b),\e])f(x)&=\e \psi(- b\nr_{D}(x))\,f(x), \quad b\in k, 
 \label{WeilR-f2}
\\
r_{\psi,D}([w,1])f(x)&={\mathcal F}_{D}f(x)\,\gamma_{\psi}(V_D,-\nr_{D}),
 \label{WeilR-f3}
\end{align}
where 
$$
\gamma_{\psi}(V_D,-\nr_{D}):=\alpha_{\psi}(1)\,\begin{cases} +1 \quad &(D\cong {\rm M}_2(k)), \\ 
-1 \quad &(\text{$D\cong {\rm H}$}.
\end{cases}
$$
We let the orthogonal group $O(V_D)$ act from the right on $V_D$. The regular representation of the orthogonal group $O(V_D)$ on $\Scal(V_D)$ is denoted by $R$, i.e., 
$$
[R(h)f](x)=f(xh), \quad h\in O(V_D),\,x\in V_D,
$$
Then $R(h)$ commutes with the operators $r_{\psi,D}(\sigma)\,(\sigma \in \widetilde S)$. Set $\hat{R}(g):=R(\rho(g))$ for $g\in {\rm PGL}_2(k)$. 

\subsection{Local theta lift}\label{sec:ThetaLift}
We have a representatoin $r_{\psi,D}\times \hat R$ of $\ol{S} \times P D^\times$ on $\Scal(V_D)$. A pair $(\tilde \pi,\pi)$ of an irreducible smooth representation $\tilde \pi$ of $\ol{S}$ and an irreducible smooth representation $\pi$ of $P D^\times$ is in theta-correspondence if there exists a non-zero bilinear map $\beta:V_{\pi}\times \Scal(V_D) \longrightarrow V_{\tilde \pi}$ such that 
\begin{align}
\beta(\pi(g)\phi,\hat R(g)\Phi)&=\beta(\phi,\Phi), \quad g\in P D^\times,\\ \beta(\phi,r_{\psi,D}(\sigma)\Phi)&=\tilde \pi(\sigma)\,\beta(\phi,\Phi), \quad \sigma\in \widetilde S
 \label{ThetaLift-f00}
\end{align}
for all $(\phi, \Phi)\in V_{\pi} \times \Scal(V_D)$; if this is the case, we write $\tilde \pi=\theta_{\psi}(\pi)$. It is known that $\beta$ is unique up to non-zero scalar multiples (\cite{RallisSchiffmann}, \cite{Waldspurger1980}). 

Recall that the multiplicative group $Z\cong k^\times$ is endowed with the measure $(1-q^{-1})^{-1}\,|z|^{-1}\d z$, so that $\vol(O^\times)=1$. For each $\xi \in k(\pi)$, let $\d\tau_{T_{\xi}}$ denote the Tamagawa measure on $T_{\xi}$ defined by some $T_\xi$-invariant algebraic $2$-form on $T_{\xi}$. Let $\d\tau_{T_\xi\bsl D^\times}$ denote the quotient measure of $\d\tau_{D^\times}$ by $\d\tau_{T_\xi}$. Let $\d\tau_{Z\bsl T_\xi}$ denote the quotient measure of $\d\tau_{T_\xi}$ by the Haar measure on $Z\cong k^\times$, which is proportional to the ``orbitally'' normalized Haar measure $\d\mu_{\xi}$ on $Z\bsl T_\xi$ defined in \S\ref{sec:LocZeta}. Let $C(\xi)>0$ be the proportionality constant, i.e., 
\begin{align}
\d\tau_{Z\bsl T_\xi}=C(\xi)\,\d\mu_\xi.
 \label{ThetaUnitary-ff}
\end{align}
Let $\{I_{\xi}\}_{\xi \in k(\pi)/(k^\times)^2}$ be a family of $P D^\times$-isomorphisms $V_{\pi} \cong \cU(\pi,\xi)$. For $u_\xi \in \cU(\pi,\xi)$ and $\Phi\in \Scal(V_D)$, define a function 
\begin{align}
{\rm t}_{\psi^\xi}(u_\xi, \Phi): \sigma \longmapsto \int_{T_\xi \bsl D^\times
} u_{\xi}(x_\xi\,\rho(g))\,[r_{\psi,D}(\sigma)R(g)\Phi](x_\xi)\,\d \tau_{T_\xi\bsl D^\times}(\dot g), \quad \sigma\in \widetilde S, 
\label{ThetaLift-f0}
\end{align}
where the integral is convergent due to the compactness of the support of $g\mapsto [r_{\psi,D}(\sigma)R(g)\Phi](x_\xi)$ and the closedness of $V_{D}(\xi)$ in $V_{D}$ (\cite[\S II-1(p.228)]{Waldspurger1991}). Formulas \eqref{WeilR-f1} and \eqref{WeilR-f3} yield  
\begin{align}
{\rm t}_{\psi^{\xi}}(u_\xi, \Phi; (1,\e)n\sigma)&=\e\,\psi^{\xi}(n)\,{\rm t}_{\psi^\xi}(u_\xi, \Phi;\sigma), \quad (n,\sigma)\in N\times \ol{S}, \label{ThetaLift-f1}
\\
{\rm t}_{\psi^\xi}(u_\xi, \mathcal F \Phi;\sigma)&=\gamma_{\psi}(V_D,-\nr_D)\,{\rm t}_{\psi^\xi}(u_\xi, \Phi;\sigma w_0) 
\label{ThetaLift-f2}
\end{align}
Set $\Theta_\psi(\pi,\xi):=\{{\rm t}_{\psi^\xi}(u_\xi, \Phi)\mid \Phi \in \Scal(V_D),\,u_\xi \in \cU(\pi,\xi)\}$, which is an $\ol{S}$-invariant subspace with respect to the right-translations by elements of $\ol{S}$. Then it is known that there exists an irreducible smooth genuine representation $\tilde \pi=\theta_{\psi}(\pi)$ of $\ol{S}$ such that $k(\tilde \pi)=k(\pi)$ and its $\psi^\xi$-Whittaker model for $\xi\in k(\pi)$ is $\Theta_{\psi}(\pi,\xi)$ (\cite[Proposition 4]{Waldspurger1991}). From now on, we suppose $\pi$ is a supercuspidal representation of ${\rm PGL}_2(k)$; then $\tilde \pi$ is a supercuspidal representation of $\ol{S}$. Let $\{l^\xi\}_{\xi\in k(\pi)}$ be a family of $\psi^{\xi}$-Whittaker functions on $V_{\tilde \pi}$. Fix an intertwiner $\beta: V_{\pi}\times \Scal(V_D) \longrightarrow V_{\tilde \pi}$ as above once and for all. Then (by the multiplcitiy one theorem), we can normalize $I_{\xi}:V_{\pi}\cong \cU(\pi,\xi$ so that 
\begin{align}
\ell^{\xi}\left(\beta(\phi,\Phi)\right)={\rm t}_{\psi^{\xi}}(I_\xi(\phi),\Phi)(e), \quad (\phi,\Phi)\in V_{\pi}\times \Scal(V_D).
\label{ThetaLift-f3}
\end{align}
holds. The central character of $\tilde \pi$, which is known by \eqref{WeilR-f1} and \eqref{ThetaLift-f1}, is $\omega_{\tilde \pi}=\chi_{\psi}|Z$. Hence, the genuine characters $\chi_{\psi}\mu^2|\cdot|^{2s}$ of $\ol{A}$ (see \S\ref{sec:charA}) all restrict to $\omega_{\pi}|Z$ (see {\bf Remark} in \S\ref{sec:FE}). Recall the local zeta-function \eqref{Def-LZF} attached to $\tilde \pi,\mu$ and the local zeta-integral \eqref{Def-LZDSpherical} attached to $\pi,\mu$. The following is the key observation which motivates us the study of \eqref{Def-LZF}. 

\begin{prop} \label{ThetaLift-L1}
 For any $\phi\in V_{\pi}$ and $\Phi \in \Scal(V_{D})$, set $v_{\phi,\Phi}=\beta(\phi,\Phi)\in V_{\tilde \pi}$. Then for $\Re(s)\gg 0$, 
\begin{align*}
{\bf Z}_{\xi}(\mu|\cdot|^s; \Phi,I_{\xi}(\phi))&=C(\xi)\,
|4\xi|^{s-3/2}\mu(-4\xi)\times \tfrac{1}{2}\,
Z(2s-1, \mu^2, l^\xi,v_{\phi,\Phi})
\end{align*}
\end{prop}
\begin{proof}
This follows from \eqref{ThetaLift-f0}, Definitions \eqref{Def-LZF} and \eqref{Def-LZDSpherical}. 
\end{proof}

The following is our main theorem on the zeta-integrals with $\pi$-spherical functions on $PD^\times$, whose novelty lies in the description of the gamma-matrix $({\bf \Gamma}(\mu|\cdot|^{s};\xi,\eta))$ in terms of our $\Gamma_{\tilde \pi,\mu}^{\xi,\eta}(s)$ (see \S\ref{sec:Gamm}):  

\begin{thm} \label{ThetaLift-P1}
Let $(\pi,V_{\pi})$ be an irreducible supercuspidal representation of ${\rm PGL}_2(k)$ and $I_{\xi}:V_{\pi}\cong \cU(\pi,x_\xi)$ for $\xi \in k(\pi)$ a family of spherical models for $\pi$. Let $\mu$ be a character of $k^\times$. Then, for any $\Phi \in \Scal(V_{D})$ and $\phi \in V_{\pi}$, 
\begin{align*}
{\bf Z}_{\xi}({\mathcal F}\Phi,I_\xi(\phi);\mu^{-1}|\cdot|^{3/2-s})=\frac{\zeta_{k}(1)^{-1}\,\alpha_{\psi}(1)\,\varepsilon_{D}}{\#(k^\times/(k^\times)^2)}\,
\sum_{\eta \in k(\pi)/(k^\times)^2} 
{\bf \Gamma}(\mu^{-1}|\cdot|^{3/2-s};\xi,\eta)\,{\bf Z}_{\eta}(\Phi,I_\eta(\phi);\mu|\cdot|^{s})
\end{align*}
where
$$
{\bf \Gamma}(\mu^{-1}|\cdot|^{3/2-s};\xi,\eta):=C(\xi)C(\eta)^{-1}\,
|4|^{2s-3/2}|\xi\eta|^{s}\,\mu(\xi\eta^{-1})\times \Gamma_{\tilde \pi, \mu^{-2}}^{\xi,\eta}(3/2-2s), 
$$
and $\varepsilon_{D}$ denotes $+1$ or $-1$ according to $D$ is split or not. The function $
{\bf \Gamma}(\mu^{-1}|\cdot|^{3/2-s};\xi,\eta)$ is an entire function of $s\in \C$. 
\end{thm}
\begin{proof} This is a direct consequence of Theorem \ref{thm-FE} and Lemma \ref{ThetaLift-L1}. \end{proof}

\medskip
\noindent
{\bf Remark}: When $\pi$ is an unramified principal series of ${\rm PGL}_2(k)$ (so that $D={\rm M}_2(k)$), then the zeta-integrals \eqref{Def-LZDSpherical} can be analized by linking them to the two variable zeta-integral associated to the prehomogenous vector space $\{X\in {\rm M}_2(k)\mid \tr(X)=0\}$ for the triangular group $\{\left(\begin{smallmatrix} * & * \\ 0 & * \end{smallmatrix}\right)\}$(\cite[Theorem 3.6]{Sa89}). Thus, an explicit form of the gamma matrix in the local functional equation is determined in this case. Our method explained in this section also works on more general principal series of ${\rm PGL}_2(k)$ giving an explicit formula of the gamma matrix; details will appear elsewhere.

\medskip

\subsection*{Acknowledgement} 
The auhtors thank Satoshi Wakatsuki for his constant interest on this work. A part of this work is a content of master's thesis of the first author written under the supervision of Professor Wakatsuki. The first author is grateful to Yuanqing Cai for informing him the papers \cite{Man84-I} and \cite{Man84-II} of Manderscheid.
%
%


\end{document}